\documentclass[]{article}

\usepackage{a4wide}
\usepackage{amsmath}
\usepackage{amsthm}
\usepackage{amsfonts}
\usepackage{amssymb}
\usepackage{mathrsfs}

\newtheorem{thm}{Theorem}
\newtheorem{property}{Property}
\newcommand{\beqn}{\begin{eqnarray}\begin{aligned}}
\newcommand{\eqn}{\end{aligned}\end{eqnarray}}
\newcommand{\fra}[2]{\textstyle{\frac{#1}{#2}}}

\begin{document}


\title{Multiplicatively closed Markov models must form\\  Lie algebras}
\author{Jeremy G Sumner\\ \small University of Tasmania, Australia\\ \small jsumner@utas.edu.au} 

\maketitle

      
      
\begin{abstract}
We prove that the probability substitution matrices obtained from a continuous-time Markov chain form a multiplicatively closed set if and only if the rate matrices associated to the chain form a linear space spanning a Lie algebra.
The key original contribution we make is to overcome an obstruction, due to the presence of inequalities that are unavoidable in the probabilistic application, that prevents free manipulation of terms in the Baker-Campbell-Haursdorff formula.
\end{abstract}


\noindent\textbf{AMS classification codes}: 	17B45; 60J27

\section{Background}

In this note we prove a result which makes explicit the requirement that a \emph{multiplicatively-closed} Markov model must form a Lie algebra (definitions will be provided).
We consider continuous-time Markov chains and work under the general assumption that a model is determined by specifying a subset of rate matrices (or rate generators).
These models are used in a wide array of scientific modelling problems and have been previously studied in the context of Lie group theory by \cite{johnson1985markov,mourad2004lie}.
Although the results given here are general, we are motivated primarily by applications to phylogenetics. 

Phylogenetics consists of the mathematical and statistical methods applied to reconstructing evolutionary history from observed molecular sequences such as DNA \cite{felsenstein2004}. 
Recent theoretical work \cite{fernandez2015lie,sumner2012lie} has discussed the relevance of Lie groups and algebras to this applied area.
The importance of Lie algebras to robust phylogenetic modelling has been demonstrated using simulation in \cite{sumner2012general}, as well as on a diverse set of biological data sets in \cite{woodhams2015new}.
The class of Markov models that form Lie algebras is discussed in the recent textbook on mathematical phylogenetics \cite{steel2016phylogeny}, and this approach also has important applications outside of phylogenetic modelling \cite{house2012lie}.
However, previous work on this topic has not established the necessity of a Lie algebra in the general setting.
In Theorem~\ref{thm1}, we establish that the Lie algebra property is a consequence of model assumptions which we claim are natural, easily understandable, and well justified in the applied setting.

Our results fit into the general theory of Lie semigroups and Lie semialgebras as developed by Hilgert and Hofmann \cite{hilgert1985semigroups}.
However, the approach we follow here gives the most direct path explicitly tailored to the practical setting of Markov chains and does so with minimal abstract theory.

\section{Main result}

Fixing notation, we denote $\mathcal{L}\subset \text{Mat}_{n}(\mathbb{R})$ as the set of real valued $n\times n$ matrices with zero-column sums and $\mathcal{L}^+\!\subset\mathcal{L}$ as the subset of matrices with non-negative off-diagonal entries.
We then have the interpretation that $Q\in \mathcal{L}$ corresponds to a valid rate matrix for a continuous-time Markov chain if and only if $Q\in \mathcal{L}^+$.
To distinguish from members of $\mathcal{L}$ we refer to the members of $\mathcal{L}^+$ as \emph{stochastic}.

The reader who prefers to use row sums for matrices associated to a Markov chain, may simply modify the definitions above appropriately and read what follows without variation.

We assume that a given model is then specified as a subset $\mathcal{R}^+\subseteq \mathcal{L}^+$ which is defined either using a parameterization, or by giving some (polynomial) constraints on the matrix entries.
In phylogenetics, the former situation is the norm and it is standard to use methods such as maximum likelihood to provide estimates of these model parameters.
However, the former specification can usually be reinterpreted using the latter --- which also plays a role in some formulations (such as the `group-based' \cite[Chap. 8]{semple2003phylogenetics} and `equivariant' \cite{draisma2009ideals} model classes).  
An example of a popular phylogenetic model will be given in the next section.
This motivates:

%
\begin{property}
\label{prop0}
A model $\mathcal{R}^+$ is expressible as an intersection $\mathcal{R}^+=\mathcal{R}\cap \mathcal{L}^+$ where  $\mathcal{R}\subseteq \mathcal{L}$ is determined by a finite set of polynomial constraints on the matrix entries of members of $\mathcal{L}$.
That is, there exist polynomials $f_1,f_2,\ldots ,f_r$ on the variables $Q\!=\!(q_{ij})$ such that $\mathcal{R}=\left\{Q\in \mathcal{L}:0=f_1(Q)\!=\!f_2(Q)\!=\!\ldots \!=\!f_r(Q)\right\}\subseteq \mathcal{L}$.
Additionally, we demand that $\mathcal{R}$ is \emph{minimal} in the sense that there is no similarly constructed set $\mathcal{R}'\subset \mathcal{R}$ such that $\mathcal{R}^+=\mathcal{R}'\cap \mathcal{L}^+$ also.
\end{property}
%

Although Property~\ref{prop0} does not  imply that $\mathcal{R}$ is necessarily unique, the minimality condition ensures that $\mathcal{R}$ does not contain any members superfluous to the determination of $\mathcal{R}^+$.
A simple example gives a clear motivating precedent for this condition, as follows.

Consider $(x,y)\in\mathbb{R}^2$ and the line $y\!=\!x$ restricted to the positive orthant:
\[
\{(x,y)\in \mathbb{R}^2:x\!-\!y=0,x,y\geq 0\}.
\]
Clearly, we most simply obtain this set  by taking the intersection of the positive orthant with the line $y\!=\!x$ (defined as the subset of $(x,y)\in \mathbb{R}^2$ satisfying the polynomial constraint $x\!-\!y=0$).
However, we may also obtain this set by taking the intersection of the positive orthant with the \emph{pair} of lines defined by the quadratic constraint $y^2\!=\!x^2$.
In this case, analogous application of Property~\ref{prop0} would ensure that we choose the former possibility.

%

Following general Markov chain theory in the time-homogeneous setting, given some amount of elapsed time $t$, the probability substitution matrix associated with $Q\in\mathcal{R}^+$ is computed as the matrix exponential $M\!=\!e^{Qt}$ (using the power series $e^A=\sum_{m\geq 0}\fra{1}{m!}A^m$).
Since $t\geq 0$ may take on any non-negative value, it is sensible to consider:

%
\begin{property}
\label{prop1}
A model $\mathcal{R}^+$ is closed under non-negative scalar multiplication.
That is, for all $Q\in \mathcal{R}^+$ and $\alpha\geq 0$, it follows that $\alpha Q\in \mathcal{R}^+$ also.
\end{property}

%

If Property~\ref{prop0} is assumed, Property~\ref{prop1} follows if and only if the polynomial constraints defining $\mathcal{R}$ are homogeneous.
Up to conventions of exactly how models are parameterized (possibly obscured by conventions of overall scaling, and `normalisation'), as far as we are aware \emph{all} phylogenetic models proposed in the literature have Property~\ref{prop1}.
When Property~\ref{prop1} is assumed we may simplify notation by writing $e^{Q}$ in place $e^{Qt}$.

We now place a third reasonable restriction on a model $\mathcal{R}^+$ by imposing, what we refer to as, multiplicative closure.
This property is relevant in any setting that generalises from the time-homogeneous to time-inhomogeneous formulation of continuous-time Markov chains.
In rough terms, what we mean by this is that, if $Q,Q'$ are in a model, then there exists another $\widehat{Q}$ also in the model such that $e^Qe^{Q'}=e^{\widehat{Q}}$.
This question rouses the BCH (Baker-Campbell-Hausdorff \cite{campbell1898}) formula for all $n\times n$ matrices $A,B$:
\[
\log(e^Ae^B)={A+B+\frac{1}{2}\left[A,B\right]+\frac{1}{12}\left(\left[A,\left[A,B\right]\right]+\left[B,\left[B,A\right]\right]\right)+\ldots}
\]
(where $\log$ is the matrix-log power series and $[A,B]\!=\!AB-BA$ is the `Lie bracket', or `commutator').
This naturally  leads to a discussion of Lie algebras in the context of continuous-time Markov chains.
Precisely how this arises is developed in the argument that follows.
Careful attention to detail must be shown however, since, for certain cases, it is possible that either (i) $\widehat{Q}$ does not belong to $\mathcal{L}^+$ or (ii) the BCH series does not converge.
The obstruction we overcome in this note is that there is no immediate means available to isolate terms in the BCH series since, by construction, a model $\mathcal{R}^+$ does not form a linear space.

The definitions and notation required to state and prove our main result are given in the following steps:

\begin{enumerate}
\item Let $\mathcal{M}$ be the semigroup generated by the set $\exp(\mathcal{R}^+)\!=\!\{e^{Q}:Q\in \mathcal{R}^+\}$.
Equivalently, $\mathcal{M}$ is the intersection of all semigroups that contain $\exp(\mathcal{R}^+)$.
(Notice $\mathcal{M}$ includes the identity matrix, since $I=e^{Q\cdot 0}$, so $\mathcal{M}$ is technically a monoid.)
\item Let $\overline{\mathcal{R}}$ be the set of (scaled) logarithms of the members of $\mathcal{M}$. 
Specifically, for what follows it is sufficient to take $\overline{\mathcal{R}}=\left\{\alpha \log(M): \alpha\geq 0, M\in \mathcal{M}\right\}$ where $\log$ is the matrix-log power series (wherever it converges).
Since $\log(e^{Q\epsilon})=Q\epsilon$ for sufficiently small $\epsilon$, we see that $\mathcal{R}^+\subseteq \overline{\mathcal{R}}$.
In general, this definition allows for the circumstance that $ \overline{\mathcal{R}}$ may contain rate matrices that are non-stochastic  and/or \emph{not members of $\mathcal{R}$} --- the latter is our crucial observation.
\end{enumerate}

We are now in a position to state our third proposed property for continuous-time Markov chains:

%

\begin{property}
\label{prop2}
A model $\mathcal{R}^+$ satisfies $\overline{\mathcal{R}}\subseteq \mathcal{R}$.
\end{property}

%

We claim that Property~\ref{prop2} is a very reasonable demand on a model since it is saying that all expressions of the form $\log(e^Qe^{Q'})=\widehat{Q}$ produce rate matrices $\widehat{Q}$ which satisfy the same constraints as the matrices $Q,Q'$ (up to  possible relaxation of the stochastic condition of membership in $\mathcal{L}^+$).
When this is the case, we say that the model is \emph{multiplicatively closed}.

\begin{thm}
\label{thm1}
Suppose a model $\mathcal{R}^+$ satisfies Property~\ref{prop0}.
Then $\mathcal{R}^+$ satisfies Properties~\ref{prop1} and~\ref{prop2} if and only if $\mathcal{R}=\text{span}_{\mathbb{R}}(\mathcal{R}^+)$ and this space forms a real Lie algebra 
\end{thm}

\begin{proof}

Assume throughout that $\mathcal{R}^+$ satisfies Property~\ref{prop0}.

\begin{itemize}
\item Assume $\mathcal{R}^+$ satisfies Properties~\ref{prop1} and~\ref{prop2}.

For $a,b\geq 0$ and $Q,Q'\in \mathcal{R}^+$ we have $aQ,bQ'\in \mathcal{R}^+$ also (by Property~\ref{prop1}).
Then $\log(e^{aQ}e^{bQ'})\!=\!aQ+bQ'+\frac{1}{2}ab\left[Q,Q'\right]+\ldots\in \overline{\mathcal{R}}$ for some choice of $a,b$ small enough such that the series converges.
By Property~\ref{prop2}, we have
$aQ+bQ'+\frac{1}{2}ab\left[Q,Q'\right]+\ldots\in \mathcal{R}$ also.
Choosing $a=b$ and rescaling by $a^{-1}$ (using Property~\ref{prop1}) we have, in the limit  $a\rightarrow 0$, $Q+Q'\in \mathcal{R}$.

Using Property~\ref{prop1}, we observe that this generalizes to $\alpha Q+\beta Q'\in \mathcal{R}$ for all $\alpha,\beta\geq 0$. 
More specifically, since $\alpha Q+\beta Q'\in \mathcal{L}^+$  and $\mathcal{R}^+\!=\!\mathcal{R}\cap \mathcal{L}^+$, we have $\alpha Q+\beta Q'\in \mathcal{R}^+$ for all $\alpha,\beta\geq 0$.
Iterating this result, shows that $\alpha_1 Q_1+\alpha_2 Q_2+\ldots +\alpha_k Q_k\in \mathcal{R}$ for all choices $Q_i\in \mathcal{R}^+$ and $\alpha_i\geq 0$.

However, since the constraints defining $\mathcal{R}$ are polynomial, this result must be true more generally for all choices $\alpha_i\in \mathbb{R}$.
Thus:
\[
\text{span}_{\mathbb{R}}(\mathcal{R^+})=\left\{\alpha_1 Q_1+\alpha_2 Q_2+\ldots +\alpha_k Q_k:Q_i\in\mathcal{R}^+,\alpha_i\in\mathbb{R}\right\}\subseteq \mathcal{R},
\] 
which yields:
\[
\mathcal{R}\cap \mathcal{L}^+=\mathcal{R}^+\subseteq \text{span}_{\mathbb{R}}(\mathcal{R^+})\cap \mathcal{L}^+\subseteq \mathcal{R}\cap \mathcal{L}^+,
\]
so equality $\mathcal{R}^+=\mathcal{R}\cap \mathcal{L}^+= \text{span}_{\mathbb{R}}(\mathcal{R^+})\cap \mathcal{L}^+$ follows, and the minimality condition demanded by Property~\ref{prop0} shows $\mathcal{R}=\text{span}_{\mathbb{R}}(\mathcal{R^+})$.

Having established that $\mathcal{R}$ is a linear space, we can now freely isolate terms in the BCH formula and be guaranteed to stay in $\mathcal{R}$. 
In particular, taking $Q,Q\in \mathcal{R}^+$ and $\epsilon> 0$, we see that 
\[
\log(e^{\epsilon Q}e^{\epsilon Q'})-\left(\epsilon Q+\epsilon Q'\right)=\textstyle{\frac{1}{2}}\epsilon^2\left[Q,Q'\right]+\ldots\in \mathcal{R},
\] 
so, rescaling by $2\epsilon^{-2}$ and taking the limit $\epsilon\rightarrow 0$, we have  $\left[Q,Q'\right]\in \mathcal{R}$ also.
Thus $\text{span}_{\mathbb{R}}(\mathcal{R}^+)=\mathcal{R}$ forms a real Lie algebra, as required.


\item
Assuming $\mathcal{R}=\text{span}_{\mathbb{R}}(\mathcal{R}^+)$ shows that the constraints defining $\mathcal{R}$ are linear, which implies Property~\ref{prop1} is satisfied. 
Further assuming $\text{span}_{\mathbb{R}}(\mathcal{R}^+)$ is a real Lie algebra and applying the BCH formula shows that each member of $\overline{\mathcal{R}}$ is a member of $\mathcal{R}$.
Hence Property~\ref{prop2} is satisfied.
\end{itemize}

\end{proof}


\section{Example}
We illustrate this process with the Hasegawa-Kishino-Yano (HKY) \cite{hasegawa1985} model  of DNA substitutions.
This is an example of a time-reversible model \cite{tavare1986}, and is defined via the parameterization (rows and columns ordered as $A,G,C,T$):
\[
Q=\left(
\begin{matrix}
\ast & \kappa\alpha_A & \alpha_A & \alpha_A \\
\kappa\alpha_G & \ast & \alpha_G & \alpha_G \\
\alpha_C & \alpha_C & \ast &  \kappa\alpha_C \\
\alpha_T & \alpha_T & \kappa\alpha_T  & \ast
\end{matrix}
\right),
\]
where the missing entries $\ast$ are chosen to ensure unit column sums.
The parameters $\alpha_i\geq 0$ are proportional to the equilibrium nucleotide frequencies of the Markov chain and $\kappa\geq 0$ is included to accommodate `transition/transversion' ratio (distinguishing substitutions within both, the `purines' $A\leftrightarrow G$ and, the `pyrimidines' $C\leftrightarrow T$, from substitutions across these two groups).

Equivalently, we may express the HKY model as the subset of rate matrices
\[
\mathcal{R}^+_{\text{HKY}}=
\left\{ Q \in \mathcal{L}^+: \begin{matrix}q_{13}\!=\!q_{14},q_{23}\!=\!q_{24},q_{31}\!=\!q_{32},q_{41}\!=\!q_{42}\\q_{12}q_{23}\!=\!q_{21}q_{13},q_{34}q_{13}\!=\!q_{12}q_{31},q_{43}q_{13}\!=\!q_{12}q_{41},\ldots\end{matrix} \right\}
,
\]
(where the displayed constraints are sufficient to determine the model).
We immediately see that $\mathcal{R}^+_{\text{HKY}}$ is not multiplicatively closed since the defining constraints are not linear.
To illustrate the issue, we give a numerical example.

We chose $Q_1,Q_2\in \mathcal{R}^+_{\text{HKY}}$ via $(\alpha_A,\alpha_G,\alpha_C,\alpha_T;\kappa)=(.02,.01,.005,.009; 1.5)$ and $(.03,.01,.006,.008;1.4)$ respectively, and computed (using Mathematica):
\[
\log(e^{Q_1}e^{Q_2})=
\left(
\begin{matrix}
 -0.0571752 & 0.0718248& 0.0498348& 
  0.0498348\\ 0.0291051& -0.0998949& 0.0200951& 
  0.0200951\\ 0.0109967& 0.0109967& -0.0947047& 
  0.0158953\\ 0.0170734& 0.0170734& 0.0247748& -0.0858252
\end{matrix}
\right).
\]

Attempting to find this matrix in the set $\mathcal{R}^+_{\text{HKY}}$, we are immediately led to \[(\alpha_A,\alpha_G,\alpha_C,\alpha_T)\!=\!(0.0498348,0.0200951,0.0109967, 0.0170734)\] but no consistent solution for $\kappa$ is obtainable (in fact four different values are required).
Therefore,  $\log(e^{Q_1}e^{Q_2})$ is not a member of $\mathcal{R}^+_{\text{HKY}}$ (or indeed $\mathcal{R}_{\text{HKY}}$ under relaxation of the stochastic conditions).


Following the definitions given in the previous section, the form obtained in this example does however suggest that all rate matrices in the closure $\overline{\mathcal{R}}_{\text{HKY}}=\log(\mathcal{M}_{\text{HKY}})$ are of the form
\[
\left(
\begin{matrix}
 \ast & \kappa_1  & \alpha &  \alpha\\ 
\kappa_2 & \ast & \beta & \beta\\ 
 \gamma & \gamma & \ast & \kappa_3 \\ 
 \delta & \delta & \kappa_4 & \ast
\end{matrix}
\right).
\]
That this is indeed the case is confirmed by two simple computations:
\begin{enumerate}
\item[(i).] $\text{span}_{\mathbb{R}}(\mathcal{R}^+_{HKY})=\left\{
\left(
\begin{matrix}
 \ast & \kappa_1  & \alpha &  \alpha\\ 
\kappa_2 & \ast & \beta & \beta\\ 
 \gamma & \gamma & \ast & \kappa_3 \\ 
 \delta & \delta & \kappa_4 & \ast
\end{matrix}
\right)
:\alpha,\beta,\gamma,\delta,\kappa_1,\ldots,\kappa_4\in \mathbb{R}
\right\},
$
\item[(ii).] This set forms a Lie algebra (in fact, this is Model 8.8 in the Lie-Markov hierarchy \cite{fernandez2015lie}). 
\end{enumerate}

Thus, the span of the HKY model forms a Lie algebra (without the additional need to take closure under Lie brackets).

\section{Discussion}

The contribution of this work was to lay out conditions on a continuous-time Markov chain (Properties~\ref{prop0},\ref{prop1},\ref{prop2}) in order to show that multiplicative closure necessitates that the associated rate matrices are minimally contained inside a Lie algebra.
The conditions need to be set up carefully in order to, firstly, be convincingly well-motivated to the applied setting and, secondly, allow for the relatively elementary proof of the main result (Theorem~\ref{thm1}). 
This result provides a solid justification (albeit post-hoc) for recent work exploring the classification, enumeration, and application of this class of Markov models  \cite{fernandez2015lie,sumner2012lie}.

We focussed on continuous-time models and hence, naturally, assumed a model is defined in terms of its rate matrices (c.f. Property~\ref{prop0}). 
This does however leave open the possibility that a Markov chain defined \emph{solely at the level of substitution matrices} could be multiplicatively closed without necessitating the existence of an associated Lie algebra (constructed as the tangent space at the identity).
We conjecture that this is not possible, but leave the details for future work.

\subsubsection*{Acknowledgement}

This research was supported by Australian Research Council (ARC) Discovery Grant DP150100088. 
I would like to thank Michael Woodhams, Barbara Holland, and the anonymous reviewers for their helpful comments on this work.

\bibliography{bch}

\end{document}